\documentclass[amsmath,amssymb,showkeys, showpacs, superscriptaddress]{revtex4-1}
\usepackage{graphicx}
\usepackage{amsmath,tikz}
\usetikzlibrary{matrix, positioning}
\usepackage{amsthm}%
\usepackage{amssymb}
\DeclareMathAlphabet{\mathpzc}{OT1}{pzc}{m}{it}
 \usepackage{xcolor}
 \newtheorem{theorem}{Theorem}[section]
\newtheorem{lemma}[theorem]{Lemma}

\theoremstyle{remark}

\theoremstyle{definition}

\renewcommand{\a}{\alpha}
\renewcommand{\b}{\beta}
\renewcommand{\l}{\lambda}

 \newcommand{\bea}{\begin{eqnarray}} 
\newcommand{\eea}{\end{eqnarray}}  
\newcommand{\bg}{\begin{gathered}}
\newcommand{\eg}{\end{gathered}}
\newtheorem{thm}{Theorem}[section]

\expandafter\let\csname equation*\endcsname\relax
\expandafter\let\csname endequation*\endcsname\relax

\newcommand{\ba}{\begin{array}}
\newcommand{\ea}{\end{array}}

  \numberwithin{equation}{section}

\newcommand{\aw}{Askey--Wilson }

\renewcommand{\a}{\alpha}
\renewcommand{\b}{\beta}
\renewcommand{\l}{\lambda}

\newtheorem{rem}[thm]{Remark}

\begin{document}



\title{A discrete and $q$ Asymptotic Iteration Method}

\author{Mourad E. H.  Ismail}
\email{mourad.eh.ismail@gmail.com. } 
\affiliation{Department of Mathematics, University of Central Florida, Orlando, Florida 32816, USA}

\author{Nasser Saad\footnote{Corresponding Author: nsaad@upei.ca }}
\affiliation{School of Mathematical and Computational Sciences,
University of Prince Edward Island, 550 University Avenue,
Charlottetown, PEI, Canada C1A 4P3.}

\begin{abstract}
We introduce a finite difference and $q$-difference analogues of the 
Asymptotic Iteration Method of Ciftci, Hall, and Saad.  We give necessary, and sufficient condition for the existence of a polynomial solution to a general linear second-order difference or $q$-difference equation subject to a ``terminating condition", which is precisely defined.  When a difference or $q$-difference equation has a polynomial solution, we show how to find the second solution. 
\vskip0.1true in
\noindent {\small 2010 \emph{Mathematics Subject Classification} Primary 39A10, 29A13; Secondary 33D99.}
\end{abstract}

\keywords{asymptotic iteration method, polynomial solutions of difference equations, $q$-difference equations.}



 \maketitle

\section{Introduction} 

The problem of finding polynomial solutions to differential equations of 
Sturm--Liouville type goes back to the 19th century. Routh \cite{Rou} 
essentially solved the problem of finding orthogonal polynomial solutions 
to differential equations of the form 
\bea\label{sec1eq1}
f(x) y^{\prime\prime}(x) + g(x) y^{\prime}(x) + h(x) y(x) = \lambda_n y(x)
\eea
where $f, g$ and $h$ are polynomials and $\lambda_n$ is the spectral 
parameter. He demanded that the (\ref{sec1eq1}) 
has polynomial solutions of degree $n$ for $n=0, 1, \cdots, N$, where 
$N$ is a fixed number $>1$, or is $+\infty$.  Earlier Heine \cite[Section 6.8]{Sze}
  considered polynomial solutions to a differential equation of the form 
\bea
\label{sec1eq2}
f(x) y^{\prime\prime}(x) + g(x) y^{\prime}(x) + h(x) y(x) =0.
\eea
where $f$, and $g$ are given polynomials of degrees at most $p+1$ and 
$p$, respectively, while $h$ is a polynomial of degree $p-1$, to be 
determined in order for equation (\ref{sec1eq2}) to have a polynomial 
solution of a prescribed degree $n$. Stieltjes, motivated by an 
electrostatic equilibrium problem 
\cite[Chapter 3]{Ismbook},  also studied this problem.  
The polynomials $h$ in (\ref{sec1eq2}) are called 
Van Vleck polynomials and the polynomial solution to (\ref{sec1eq2}) is 
called a Stieltjes polynomials. This theory is well-explained in Section 9 of 
Marden's excellent monograph \cite{Mar}.  When $f$ has real and simple 
zeros which interlace with the zeros of $g$ the theory further simplifies, 
see \S 6.8 in \cite{Sze}. 
In 1929  Bochner \cite{Boc} characterized all polynomial solutions 
(not necessarily orthogonal) to \eqref{sec1eq1} with $N = \infty$. 
Routh's theorem was extended to the difference, or $q$-difference 
operators, see the survey article  \cite{Als}. A more general treatment 
is in Chapter 20 of \cite{Ismbook}, where the corresponding problem for the 
\aw operator is also mentioned.  A recent variation on the Routh 
(or Bochner) problem was introduced in the works 
\cite{Gom:Kam:Mil1}--\cite{Gom:Kam:Mil3} by D. G\'{o}mez-Ullate, 
N. Kamran, R. Milson. They looked for equations of the type 
(\ref{sec1eq1}) but they demanded them to have orthogonal polynomials 
solutions of degree $n$, for all $n \ge m$ for some $m$. This 
investigation  generated what is now called exceptional orthogonal 
polynomials. 
\vskip0.1true in
The Asymptotic Iteration Method (AIM) was introduced in 
2003 in 
\cite{Cif:Hal:Saa}, \cite{Saa:Hal:Cif}, see also \cite{Saa:Hal},  as a 
tool to find closed form solution to a fairly large class of second-order 
differential equations.  The method has been applied to a variety of 
problems and seems to provide new insight into an old problem 
\cite{Znojil}. 

\vskip0.1true in
 We felt that working out a discrete and a $q$-analogue of AIM 
is a worthwhile endeavor 
and this paper indeed provides a discrete and a $q$-
analogue of AIM, which we refer to as DAIM and $q$-AIM. The 
techniques used in both cases are almost parallel, so we included a 
detailed treatment of DAIM but only sketched the outline of $q$-AIM. 
We give some  examples to illustrate the power of this approach. 

\vskip0.1true in
Section 2 contains a brief list of definitions and the 
notations used in this 
work.  In Section 3 we introduce the discrete version of AIM, called DAIM. 
In it, we show how to construct two linearly independent solutions of a 
general linear second order difference equation with variable 
coefficients under  the assumption  (\ref{sec3eq6}), which we shall 
call a terminating condition. 
 In Section 4 we prove that the general linear second order difference 
 equation has a polynomial solution if and only if the terminating 
 condition (\ref{sec3eq6}) holds for some $n$.  In Section 5  we give several examples including Euler-type equations and the discrete version of the hypergeometric equation.  Section 6 treats the linear second-order $q$-difference 
 equations where we derive the theory $q$AIM in parallel with the DAIM 
 technique.  We also characterize $q$-difference equations which 
 have a polynomial solution regarding a terminating condition. Section 7 
we implement the $q$-AIM technique
  to  explore several examples including the $q$-Laguerre difference equation, Al-Salam-Carlitz $q$-difference equation, and the Stieltjes-Wigert $q$-difference equation.  Section 8 discusses the limitations 
  of the AIM, DAIM, and $q$-AIM method. 
\vskip0.1true in
It is worth mentioning that the Heine and Stieltjes theories 
for differential 
equations with polynomial solutions have not been extended to 
the difference or $q$-difference equations. It will be interesting to develop such a theory.
\vskip0.1true in
\begin{rem}\label{rem1}
The difference or $q$-difference equations we consider have parameters. One important point is that it may be easy to find necessary conditions on the parameters in order for the equation to have a polynomial solution. Our approach gives necessary and sufficient conditions for a polynomial solution to exist. 
\end{rem}

\section{Preliminaries for  difference and $q$-difference equations}
\noindent 
 
\noindent  
 It easy to see that the  problem 
\begin{equation}\label{sec2eq1}
y(n+1)=\lambda(n)y(n)+g(n),\qquad y(n_0)=y_0,\qquad n\ge n_0\ge 0.
\end{equation}
when $\lambda(n) \ne 0$ for all $n$, has  the solution  
\begin{align}\label{sec2eq2}
y(n)=\left(\prod_{i=n_0}^{n-1} \lambda(i)\right)y_0+
\sum_{i=n_0}^{n-1}\left[\left(\prod_{\ell=i+1}^{n-1}
\lambda(\ell)\right)g(i)\right].
\end{align}

\noindent We shall use the standard notation for the finite difference 
operators  $E, \Delta, \nabla$ as in \cite{Jor}, \cite{Mil}.  
 In general, for $n=1,2,\dots,$  we have 
\bea
\label{sec2eq3}
\Delta^n f(x) &=& ((E- I)^n f)(x)=\sum_{k=0}^n(-1)^k {n\choose k} f(x+n-k),
   \\
\nabla^n f(x)&=&((I- E^{-1})^n f)(x)=  \sum_{k=0}^n(-1)^k {n\choose k} 
f(x-k).
\label{sec2eq4}
\eea
Some of the formulas used in the sequel are:
\begin{align}
\nabla^k f(x+k)&=\Delta^k f(x),\qquad \Delta^k f(x-k)=\nabla^k f(x), 
\quad k=1, 2, \cdots, \label{sec2eq5}\\
\Delta \nabla f(x)&=\nabla\Delta f(x) = f(x+1)-2f(x)+f(x-1)= (\Delta-  \nabla) f(x).\label{sec2eq6}
\end{align}
 
\vskip0.1true in
\noindent  The product rule is 
\bea\label{sec2eq7}
\Delta[f(x)g(x)]
&=& g(x) \Delta f(x)+f(x+1)\Delta g(x)\\
&=&f(x)\,\Delta g(x)+g(x)\,\Delta f(x)+\Delta f(x)\,\Delta g(x).
\label{sec2eq8}
\eea
The quotient rule is 
\bea
\label{sec2eq9}
\Delta \left(\frac{g(x)}{f(x)}\right) =\frac{f(x)\Delta g(x)-
g(x)\Delta f(x)}{f(x)f(x+1)}.
\eea
The symmetric  Leibniz rule for finite difference operators is \cite{Max}
\bea\label{sec2eq10}
(\Delta^n fg)(x) = n!  \sum_{j, k\ge 0,~ j+k \le n} \frac{(\Delta^jf)(x) 
(\Delta^k g)(x)}{j!\,k!\,(n-j-k)!}. 
\eea
 \vskip0.1true in
\noindent The notation for $q$-shifted factorials  is \cite{Gas:Rah}, \cite{And:Ask:Roy} 
 \bea\label{sec2eq11}
 (a;q)_0 := 1,\quad (a;q)_n = \prod_{j=0}^{n-1} (1-a\,q^j), \quad n =1, 2, 
 \cdots, \qquad\textup {or}\quad\; \infty.
 \eea
 Here we always assume that $0 < q <1$.  The $q$-analogue of the 
 binomial coefficient is 
 \bea\label{sec2eq12}
{n \brack k}_q := \frac{(q;q)_n}{(q;q)_k (q;q)_{n-k}}.
 \eea
  We also have 
  \bea\label{sec2eq13}
  \label{eqdan}
 (1-q)^n x^n  (D_q^nf)(x) = q^{-{n\choose 2}} \sum_{k=0}^n 
 {n\brack k}_q (-1)^k q^{{k\choose 2}} f(xq^{n-k}). 
 \eea
The product and quotient rules are 
\begin{align}
D_q[ f(x)g(x)]&=g(x)D_q f(x)+f(qx)D_q g(x), 
\label{sec2eq14}\\ 
D_q \left(\frac{f(x)}{g(x)}\right)&=\frac{g(x)D_qf(x)-f(x)D_qg(x)}{g(qx)g(x)}.  \label{sec2eq15}
\end{align}
Let $\a(x)$ be continuous at $x=0$.  Then the  solution to 
 \bea
 \label{sec2eq16}
 (D_q y)(x) = \a(x) y(x),
 \eea
 which is continuous at $x=0$ is 
 \bea
 \label{sec2eq17}
 y(x) = \dfrac{y(0)}{\prod_{k=0}^\infty [1 - (1-q)\,x\,q^k \a(xq^k)]}.
 \eea
 This follows trivially. Moreover if $\a(x)$ and $\b(x)$ are continuous at $x=0$ then the solution to 
 \bea
 \label{sec2eq18}
 (D_q y)(x) = \a(x) y(x)+ \b(x), 
 \eea
 which is continuous at $x=0$, is given by 
 \bea
 \label{sec2eq19}
 y(x) = \frac{y(0)}{\prod_{k=0}^\infty [1 - (1-q)xq^k \a(xq^k)]} + \sum_{k=0}^\infty \frac{xq^k(1-q) \b(xq^k)} 
 {\prod_{j=0}^{k} [1 - (1-q)xq^j \a(xq^j)]}. 
 \eea
  If $y(x)$ satisfies a linear homogeneous difference equation then 
 $f(x)y(x)$ will satisfy the same equation if $f$ is unit periodic, that is 
 $f(x+1) = f(x)$. Thus unit periodic functions play the role 
 played by constants in the theory of differential equations. Similarly functions satisfying $f(qx) = f(x)$  play the role of constants in the theory of $q$-difference equations.  
 
\section{Discrete Asymptotic Iteration Method (DAIM)}
The second-order difference equation may take one of the following forms
\begin{eqnarray}
\Delta^2 y(x)&=\lambda_0(x)\Delta y(x)+s_0(x) y(x), \label{sec3eq1}\\
 \Delta\nabla y(x)&=\alpha_0(x)\Delta y(x)+\beta_0(x)y(x),\label{sec3eq2}
\\
\nabla\Delta y(x)&=\alpha_1(x)\nabla y(x)+\beta_1(x)y(x),
\label{sec3eq3}\end{eqnarray}
These forms are equivalent and we shall focus our attention on  the first form 
(\ref{sec3eq1}).  
\vskip0.1true in
\noindent Unlike the original form of AIM where the boundary conditions 
contributed in setting up the asymptotic solution, in the discrete version  
the initial conditions must be incorporated within the development of the analytic solution at later stage.  

 \begin{thm} \label{thm3.1}
 If $y(x)$ satisfies (\ref{sec3eq1}), then 
\begin{eqnarray}\label{sec3eq4}
\Delta^{n+2} y(x) =\lambda_{n}(x)\Delta y(x)+s_{n}(x) y(x),
\end{eqnarray}
where
\begin{eqnarray}\label{sec3eq5}
\begin{gathered}
\lambda_n(x)=\Delta \lambda_{n-1}(x)+\lambda_{n-1}(x+1)\lambda_0(x)+s_{n-1}(x+1), \quad n >0,\\
s_n(x)=\Delta s_{n-1}(x)+ \lambda_{n-1}(x+1)s_0(x), \quad n >0.
\end{gathered}
\end{eqnarray}
 \end{thm}    
 \vskip0.1true in
 \begin{proof}
\noindent The proof is by induction on $n$. 
\end{proof}

\noindent  We note that the above mentioned construction is reminiscent of the 
construction of the Lommel polynomials from the three-term recurrence 
relation of the Bessel functions given in Watson \cite{Wat} and is 
reproduced in \cite{Ismbook}. The $q$-Lommel polynomials associated with 
$J_\nu^{(2)}$ was given in \cite{Ism1} while the construction associated 
with $J_\nu^{(3)}$ was given in \cite{Koe:Swa}.

\begin{thm}
Let $\lambda_n$ and $s_n$ be as in (\ref{sec3eq5}), and set $\delta_n(x) = \lambda_n(x)\,s_{n-1}(x)-\lambda_{n-1}(x)\,s_{n}(x)$. 
If $\delta_n(x)=0$, then $\delta_{m}(x)=0$ for all $m \ge n$.
\end{thm}
\begin{proof}
It suffices to show 
that if $\delta_n(x)=0$, then $\delta_{n+1}(x)=0$. Using the definition 
(\ref{sec3eq5}),  we  find that 
\begin{align*}
\delta_{n+1}(x)
&= \lambda_{n+1}(x)s_n(x)-\lambda_n(x)s_{n+1}(x)\\
&=
s_{n}(x) \Delta \lambda_{n}(x) -\lambda_{n}(x)\,\Delta s_{n}(x)+\lambda_{n}(x+1)\,s_{n}(x)\lambda_0(x)+s_{n}(x+1)\,s_{n}(x)-\lambda_{n}(x) \lambda_{n}(x+1)s_0(x)\\
&=\left(\dfrac{s_{n}(x) \Delta \lambda_{n}(x)-\lambda_{n}(x)\,\Delta s_{n}(x)}{s_n(x)s_{n}(x+1)}\right) s_n(x)s_{n}(x+1)+\lambda_{n}(x+1)(s_{n}(x)\lambda_0(x)-\lambda_0(x)s_0(x))+s_{n}(x+1)\,s_{n}(x)\\
&=\Delta \left(\dfrac{\lambda_{n}(x)}{s_n(x)}\right)s_n(x)s_{n}(x+1)+\lambda_{n}(x+1)\,s_{n}(x)\lambda_0(x)+s_{n}(x+1)\,s_{n}(x)-\lambda_{n}(x) \lambda_{n}(x+1)s_0(x)\\
&=s_n(x)s_{n}(x+1)\bigg(\Delta \left(\dfrac{\lambda_{n}(x)}{s_n(x)}\right)+1+\dfrac{\lambda_{n}(x+1)}{s_{n}(x+1)}\lambda_0(x)-\dfrac{\lambda_{n}(x) \lambda_{n}(x+1)}{s_n(x)s_{n}(x+1)}s_0(x)\bigg)\\
&=s_n(x)s_{n}(x+1)\left(\Delta \left(\dfrac{\lambda_{n-1}(x)}{s_{n-1}(x)}\right)+1+\dfrac{\lambda_{n-1}(x+1)}{s_{n-1}(x+1)}\left(\lambda_0(x)-\dfrac{\lambda_{n-1}(x) }{s_{n-1}(x)}s_0(x)\right)\right)\\
&=s_n(x)s_{n}(x+1)\left(\dfrac{s_{n-1}(x)\Delta \lambda_{n-1}(x)-\lambda_{n-1}(x)\Delta s_{n-1}(x)}{s_{n-1}(x)s_{n-1}(x+1)}+1+\dfrac{\lambda_{n-1}(x+1)}{s_{n-1}(x+1)}\left(\lambda_0(x)-\dfrac{\lambda_{n-1}(x) }{s_{n-1}(x)}s_0(x)\right)\right)\\
&=s_n(x)s_{n}(x+1)\left(\dfrac{\Delta \lambda_{n-1}(x)+\lambda_{n-1}(x+1)\lambda_0(x)+s_{n-1}(x+1)}{s_{n-1}(x+1)}-\dfrac{\lambda_{n-1}(x)(\Delta s_{n-1}(x)+\lambda_{n-1}(x+1)s_0(x))}{s_{n-1}(x)s_{n-1}(x+1)}\right)\\
&=s_n(x)s_{n}(x+1)\left(\dfrac{\lambda_{n}(x)}{s_{n-1}(x+1)}-\dfrac{\lambda_{n-1}(x) s_{n}(x)}{s_{n-1}(x)s_{n-1}(x+1)}\right)\\
&=s_n(x)s_{n}(x+1)\left(\dfrac{s_{n-1}(x)\lambda_{n}(x)-\lambda_{n-1}(x) s_{n}(x)}{s_{n-1}(x)s_{n-1}(x+1)}\right)=0.
\end{align*}
This completes the proof. 
\end{proof}

\noindent At this stage we make the assumption that 
\begin{eqnarray}\label{sec3eq6}
\dfrac{s_{n}(x)}{\lambda_n(x)} =\dfrac{s_{n-1}(x)}{\lambda_{n-1}(x)},
\end{eqnarray}
holds for some $n$, hence for all the subsequent $n$'s. 
 \begin{thm}\label{Thm3.3}
A solution of the difference equation
\begin{align*}
\Delta^2 y(x)=\lambda_0(x)\Delta y(x)+s_0(x) y(x),
\end{align*}
 is given by
  \bea \label{sec3eq7}
  y(x)=\left(\prod_{i=x_0}^{x-1} \left[1-\dfrac{s_{n-1}(i)}{\lambda_{n-1}(i)\,} \right]\right), \qquad x=0,1,2,\dots,
  \eea
  provided that 
 $$\frac{s_{n}(x)}{\lambda_n(x)} =\dfrac{s_{n-1}(x)}{\lambda_{n-1}(x)},
$$
where $\lambda_n(x)$ and $s_n(x)$ are given by (\ref{sec3eq5}). 
\end{thm} 

\begin{proof}  Assume that $y$ is defined by (\ref{sec3eq7}). Then  
\begin{equation}\label{sec3eq8}
\dfrac{\Delta y(x)}{y(x)}=-\frac{s_{n-1}(x)}{\lambda_{n-1}(x)}.
\end{equation}
Applying $\Delta$ to (\ref{sec3eq8}) and use the quotient rule 
(\ref{sec2eq15}) we   conclude that          
\begin{eqnarray*}
 \frac{\Delta^2 y(x)}{y(x+1)}-\left(\dfrac{\Delta y(x)}{y(x)}\right)^2\dfrac{y(x)}{y(x+1)}&=-\dfrac{\Delta s_{n-1}(x)}{\lambda_{n-1}(x+1)}+\dfrac{s_{n-1}(x)\Delta \lambda_{n-1}(x)}{\lambda_{n-1}(x)\lambda_{n-1}(x+1)},
\end{eqnarray*}
which is equivalent to 
\begin{equation}
\label{sec3eq9}
\Delta^2 y(x)-\left(\dfrac{s_{n-1}(x)}{\lambda_{n-1}(x)}\right)^2 y(x)
=\left(\dfrac{s_{n-1}(x)\Delta \lambda_{n-1}(x)-\lambda_{n-1}(x)\Delta s_{n-1}(x)}{\lambda_{n-1}(x)\lambda_{n-1}(x+1)}\right)y(x+1).
\end{equation}
Using the recursive DAIM sequences (\ref{sec3eq5}) we find that  
\begin{align}\label{sec3eq10}
s_{n-1}(x)\Delta \lambda_{n-1}(x)-
\lambda_{n-1}(x)\Delta s_{n-1}(x)&=
-s_{n-1}(x)\lambda_{n-1}(x+1)\lambda_0(x) -s_{n-1}(x)s_{n-1}(x+1)
\notag\\
&+\lambda_{n-1}(x) \lambda_{n-1}(x+1)s_0(x),
\end{align} 
Now equation (\ref{sec3eq9}) becomes 
\begin{align*}
\Delta^2 y(x)
&=\left(s_0(x)-\dfrac{s_{n-1}(x)\lambda_0(x)}{\lambda_{n-1}(x)}-\dfrac{s_{n-1}(x)s_{n-1}(x+1)}{\lambda_{n-1}(x)\lambda_{n-1}(x+1)}\right)\Delta y(x)\\
&+ \left(s_0(x)-\dfrac{s_{n-1}(x)\lambda_0(x)}{\lambda_{n-1}(x)}-\dfrac{s_{n-1}(x)s_{n-1}(x+1)}{\lambda_{n-1}(x)\lambda_{n-1}(x+1)}+\left(\dfrac{s_{n-1}(x)}{\lambda_{n-1}(x)}\right)^2\right)y(x),
\end{align*}
which  can be written as
\begin{align*}
\Delta^2 y(x)
&=\lambda_0(x) 
\Delta y(x)+s_0(x) y(x)\\
&+\left(s_0(x)-\lambda_0(x)-\dfrac{s_{n-1}(x)\lambda_0(x)}{\lambda_{n-1}(x)}-\dfrac{s_{n-1}(x)s_{n-1}(x+1)}{\lambda_{n-1}(x)\lambda_{n-1}(x+1)}\right)\Delta y(x)\\
&+ \left(-\dfrac{s_{n-1}(x)\lambda_0(x)}{\lambda_{n-1}(x)}-\dfrac{s_{n-1}(x)s_{n-1}(x+1)}{\lambda_{n-1}(x)\lambda_{n-1}(x+1)}+\left(\dfrac{s_{n-1}(x)}{\lambda_{n-1}(x)}\right)^2\right)y(x),
\end{align*}
Thus, to show that $\Delta^2 y(x) - \lambda_0(x)\Delta y(x)- s_0(x) y(x) =0,$
we need to show that 
\begin{eqnarray*}
&\left(s_0(x)-\lambda_0(x)-\dfrac{s_{n-1}(x)\lambda_0(x)}{\lambda_{n-1}(x)}-\dfrac{s_{n-1}(x)s_{n-1}(x+1)}{\lambda_{n-1}(x)\lambda_{n-1}(x+1)}\right)\Delta y(x)\\
&=- \left(-\dfrac{s_{n-1}(x)\lambda_0(x)}{\lambda_{n-1}(x)}-\dfrac{s_{n-1}(x)s_{n-1}(x+1)}{\lambda_{n-1}(x)\lambda_{n-1}(x+1)}+\left(\dfrac{s_{n-1}(x)}{\lambda_{n-1}(x)}\right)^2\right)y(x). 
\end{eqnarray*}
Using (\ref{sec3eq9}) we see that we need to show that 
\begin{eqnarray*}
 \begin{gathered}
\left(-\lambda_0(x)-\dfrac{s_{n-1}(x)\lambda_0(x)}{\lambda_{n-1}(x)}-\dfrac{s_{n-1}(x)s_{n-1}(x+1)}{\lambda_{n-1}(x)\lambda_{n-1}(x+1)}+s_0(x)\right)\Delta y(x) \\+ \left(\lambda_0(x)+\dfrac{ s_{n-1}(x+1)}{ \lambda_{n-1}(x+1)}-\dfrac{s_{n-1}(x)}{\lambda_{n-1}(x)}\right)\Delta y(x) =0, 
\end{gathered}
\end{eqnarray*}
which is equivalent to showing that 
\begin{eqnarray*}
s_0(x)-\frac{s_{n-1}(x)\lambda_0(x)}{\lambda_{n-1}(x)}
-\frac{s_{n-1}(x)s_{n-1}(x+1)}{\lambda_{n-1}(x)\lambda_{n-1}(x+1)}
+ \frac{ s_{n-1}(x+1)}{ \lambda_{n-1}(x+1)}-\frac{s_{n-1}(x)}
{\lambda_{n-1}(x)} =0.  
\end{eqnarray*}
Multiply the above equality by $\lambda_{n-1}(x)\lambda_{n-1}(x+1)$ and 
apply (\ref{sec3eq10}) to reduce the problem to 
\bea\nonumber
s_{n-1}(x)\Delta \lambda_{n-1}(x)&-\lambda_{n-1}(x)\Delta s_{n-1}(x)+ s_{n-1}(x+1) \lambda_{n-1}(x)-s_{n-1}(x)\lambda_{n-1}(x+1) =0,\nonumber
\eea
which is obviously true. 
\end{proof}

\noindent We now assume that there is an $n$  such that (\ref{sec3eq6}) holds. 
In this case 
\begin{eqnarray}\label{sec3eq11}
\frac{\Delta^{n+2}y(x)}{\Delta^{n+1} y(x)}=\dfrac{\lambda_{n}(x)
\Delta y(x)+s_{n}(x) y(x)}{\lambda_{n-1}(x)\Delta y(x)+s_{n-1}(x) y(x)}
=\frac{\lambda_{n}(x)}{\lambda_{n-1}(x)}.
\end{eqnarray}
This implies
\bea
\label{sec3eq12}
\Delta^{n+1}y(x)=\Delta^{n+1} y(0)\prod_{k=0}^{x-1}\left[1+
\dfrac{\lambda_{n}(k)}{\lambda_{n-1}(k)}\right].
\eea
This is the exact analogue of equation (2.10) in \cite{Cif:Hal:Saa}. Note that (\ref{sec3eq12}) implies 
\begin{equation}\label{sec3eq13}
\Delta^{n+1}y(x+m)=\Delta^{n+1} y(x)\prod_{k=0}^{m-1}\left[1+
\dfrac{\lambda_{n}(x+k)}{\lambda_{n-1}(x+k)}\right].
\end{equation}
Using Theorem (\ref{thm3.1}) we find that  the solution to the difference equation
\begin{eqnarray*}
\Delta^2 y(x)=\lambda_0(x)\Delta y(x)+s_0(x) y(x),
\end{eqnarray*}
solves  the first-order inhomogeneous difference equation
\begin{equation}\label{sec3eq14}
\Delta^{n+1} y(x)\prod_{k=0}^{m-1}\left[1+
\dfrac{\lambda_{n}(x+k)}{\lambda_{n-1}(x+k)}\right]=\lambda_{n-1}(x+m)\,\Delta y(x+m)+s_{n-1}(x+m)\,y(x+m),
\end{equation}
namely, for $m=0,1,2,\dots$,
\begin{align}\label{sec4eq15}
\Delta y(x+m)+\dfrac{s_{n-1}(x+m)}{\lambda_{n-1}(x+m)\,}\,y(x+m)=
\dfrac{\Delta^{n+1} y(x)}{\lambda_{n-1}(x+m)\,}\prod_{k=0}^{m-1}\left[1+
\dfrac{\lambda_{n}(x+k)}{\lambda_{n-1}(x+k)}\right].
\end{align}
Comparing this with  (\ref{sec2eq1}) and (\ref{sec2eq2}) and replacing 
$\Delta^{n+1}y(x)$ by its value from (\ref{sec3eq12})  
 we see that the general solution, using $y(x)=y(x-m+m)$ is given by 
\begin{align}\label{sec3eq16} 
y(x)&=C_2\prod_{i=n_0}^{x-1}\left(1-\dfrac{s_{n-1}(i)}{\lambda_{n-1}(i)\,} \right)\notag\\
&+C_1\sum_{i=n_0}^{x-1}\left(\prod_{\ell=i+1}^{x-1}\left(1-\dfrac{s_{n-1}(\ell)}{\lambda_{n-1}(\ell)}\right)\dfrac{
\left(\prod\limits_{j=n_0}^{i-m-1}\left(1+\dfrac{\lambda_{n}(j)}{\lambda_{n-1}(j)}\right)\right)}{\lambda_{n-1}(i)\,}\prod_{k=0}^{m-1}\left[1+
\dfrac{\lambda_{n}(i-m+k)}{\lambda_{n-1}(i-m+k)}\right]\right).
\end{align} 

\begin{thm}\label{th34}
The general solution to (\ref{sec3eq1}) is given by (\ref{sec3eq16}), where $C_1$ and $C_2$ are unit periodic functions provided that \eqref{sec3eq6} is satisfied.
\end{thm}
\begin{proof}
The analysis before this theorem shows that (\ref{sec3eq16}) gives a 
solution of  (\ref{sec3eq1}). So, we only need to show that the 
coefficients of $C_1$ and $C_2$, say $y_1(x)$ and $y_2(x)$ are linear independent. This holds if and only if the Casorati determinant 
\begin{equation}\label{sec3eq17}
\left|\begin{array}{cc}
y_1(x) &  y_1(x+1) \\
y_2(x) & y_2(x+1)
\end{array}\right|, 
\end{equation}
does not vanish, which is an easy exercise. 
\end{proof} 
\section{A Criterion for Polynomial Solutions}
 \noindent The main results  of this section are Theorems \ref{thm4.1}-\ref{thm4.2} 
 which, respectively,  give necessary, and sufficient conditions for a 
 second order linear difference equation to have a polynomial solution. 
 \begin{thm}\label{thm4.1}  If the second-order difference equation 
$\Delta^2 y(x)=\lambda_0(x)\Delta y(x)+s_0(x)y(x)$  
has a polynomial solution of degree $n$, then
\begin{eqnarray*}
s_n(x)\lambda_{n-1}(x)-s_{n-1}(x)\lambda_{n}(x)=0,
\end{eqnarray*}
where
\begin{align*}
\lambda_n(x)&=\Delta \lambda_{n-1}(x)+\lambda_{n-1}(x+1)\lambda_0(x)+s_{n-1}(x+1),\\
s_n(x)&=\Delta s_{n-1}(x)+ \lambda_{n-1}(x+1)s_0(x).
\end{align*}
\end{thm}

\begin{proof} We apply (\ref{sec3eq4}) and the recursions in (\ref{sec3eq5}) to find that 
\bea\nonumber
s_n(x)\Delta^{n+1} y(x) =s_n(x)\lambda_{n-1}(x)\Delta y(x)+s_n(x)s_{n-1}(x) y(x),\\
s_{n-1}(x)\Delta^{n+2} y(x) =s_{n-1}(x)\lambda_{n}(x)\Delta y(x)+s_{n-1}(x)s_{n}(x) y(x),\label{sec4eq1}
\eea
which  then yields   
\begin{eqnarray}\label{sec4eq2}
s_n(x)\Delta^{n+1} y(x)-s_{n-1}(x)\Delta^{n+2} y(x)  &=(s_n(x)\lambda_{n-1}(x)-s_{n-1}(x)\lambda_{n}(x))\Delta y(x),
\end{eqnarray}
If  $y(x)$  
is a polynomial of degree $n$ then $\Delta^{n+1} y(x)= 
\Delta^{n+2} y(x)=0$ and the theorem follows. 
\end{proof}
\noindent The next theorem provides a converse to Theorem \ref{thm4.1}.
\begin{thm}\label{thm4.2}
If $s_n(x)\lambda_{n-1}(x)\neq 0$ and ${\lambda_{n-1}(x)}s_{n}(x)- \lambda_{n}(x)s_{n-1}(x)=0$,
then the difference equation
$\Delta^2 y(x)=\lambda_0(x)\Delta y(x)+s_0(x)y(x)$  
has a polynomial solution whose degree is at most $n$.
\end{thm}
\begin{proof}
When  $
s_n(x)\lambda_{n-1}(x)-s_{n-1}(x)\lambda_{n}(x)=0,$ 
Equation (\ref{sec4eq2}) reduces  to
\begin{eqnarray}\label{sec4eq3}
s_n(x)\Delta^{n+1} y(x)-s_{n-1}(x)\Delta^{n+2} y(x)  =0
\end{eqnarray}
which yields
\begin{eqnarray}\label{sec4eq4}
s_n(x)\Delta^{n+1} y(x)=s_{n-1}(x)\left(\lambda_{n}(x)\Delta y(x)+s_{n}(x) y(x)\right)=s_{n-1}(x)y(x)\left(\lambda_{n}(x)\dfrac{\Delta y(x)}{y(x)}+s_{n}(x)\right).
\end{eqnarray}
Let $y(x)$ be the solution given by  (\ref{sec3eq7}) then apply 
(\ref{sec3eq8})  to establish 
\begin{eqnarray*}
s_n(x)\Delta^{n+1} y(x)=s_{n-1}(x)y(x)\left(- \lambda_{n}(x)\dfrac{s_{n-1}(x)}{\lambda_{n-1}(x)}+s_{n}(x)\right)=\dfrac{s_{n-1}(x)}{\lambda_{n-1}(x)}y(x)\left({\lambda_{n-1}(x)}s_{n}(x)- \lambda_{n}(x)s_{n-1}(x)\right). 
\end{eqnarray*}
Therefore 
\bea
\nonumber
\Delta^{n+1} y(x)  
=\dfrac{s_{n-1}(x)}{s_n(x)\lambda_{n-1}(x)}y(x)\left({\lambda_{n-1}(x)}s_{n}(x)- \lambda_{n}(x)s_{n-1}(x)\right)=0.
\eea
This shows  that  $y(x)$  is a polynomial of degree at most  $n$.
\end{proof}
\section{ Examples}
\subsection{An equation of Euler type }
\noindent Consider the equation
\begin{eqnarray}\label{sec5eq1}
\Delta^2 y(x)=\frac{2(a-1)}{1+x}\Delta y(x)+\frac{a(1-a)}{x(1+x)} y(x).
\end{eqnarray}
Before applying DAIM to \eqref{sec5eq1} we explain the relevance of Remark \ref{rem1}. If $y=x^n+\text{lower order terms}$, then $x^2 \Delta^2y-n(n-1)x^n$ and $x\Delta y-n x^n$ are polynomials of degree at most $n-1$. Substituting $y=x^n+\text{lower order terms}$ in \eqref{sec5eq1} 
and equating coefficients  of $x^n$ establishes the condition $n(n-1)=2n(a-1)+a(1-a)$, which implies $a=n,n+1$. These are necessary conditions. 

We now apply DAIM with
\begin{align}\label{sec5eq2}
\lambda_0(x)=\frac{2(a-1)}{1+x},\qquad s_0(x)= \frac{a-a^2}{x(1+x)}. 
\end{align}
From the DAIM sequences (\ref{sec3eq5}), we note that
\begin{equation}\label{sec5eq3}
\lambda_1(x)=\frac{3 (a-2) (a-1)}{(1+x) (2+x)},\quad s_1(x)=-\frac{2 (a-2) (a-1) a}{x (1+x) (2+x)}
\end{equation}
and after computing the first few $\lambda_n$'s and $s_n$'s we use 
induction to show that for arbitrary $n$, we have
\begin{equation}\label{sec5eq4}
\begin{gathered}
\lambda_n(x)=\frac{(n+2) \prod_{k=0}^n (a-k-1)}{\prod_{k=0}^{n} (x+k+1)},\quad s_n(x)=-\frac{(n+1)\,a \prod_{k=0}^n (a-k-1)}{\prod_{k=0}^{n+1}(x+k)}.
\end{gathered}
\end{equation}
We then conclude that
\bea\label{sec5eq5}
\begin{gathered}
\delta_n(x)=\lambda_n(x)s_{n-1}(x)-\lambda_{n-1}(x)s_n(x)=-\frac{ \, a\, (1-a)_n (1-a)_{n+1}}{(x)_{n+1} (x+1)_{n+1}}.
\end{gathered}
\eea
Thus $ \delta_n(x)=0$ if $a=n+1$. 
To construct the exact solution where $a=n+1$, we apply (\ref{sec3eq7}) 
and find that
\begin{equation}\label{sec5eq6}
y_n(x)=\left(\prod_{i=x_0}^{x-1}
 \left[1+\frac{n}{i}\right]\right)=\dfrac{(x)_n}{(x_0)_n},\quad n=0,1,2,\dots. 
\end{equation}
To find a second independent solution, we shall use two different approaches, first using the second independent solution as given by equation \eqref{sec3eq16} with $a=n+1$, $m\equiv n=1,2,\cdots$,
\begin{align*} 
y_2(x)&=\sum_{i=n_0}^{x-1}\left(\prod_{\ell=i+1}^{x-1}\left(1-\dfrac{s_{n-1}(\ell)}{\lambda_{n-1}(\ell)}\right)\dfrac{
\left(\prod\limits_{j=n_0}^{i-m-1}\left(1+\dfrac{\lambda_{n}(j)}{\lambda_{n-1}(j)}\right)\right)}{\lambda_{n-1}(i)\,}\prod_{k=0}^{m-1}\left[1+
\dfrac{\lambda_{n}(i-m+k)}{\lambda_{n-1}(i-m+k)}\right]\right)\\
&=\sum _{i=n_0}^{x-1} \frac{(-1)^{1+n} 2^{n_0-n} \Gamma (i+n+1) (n+1)_{x-1} }{n (n+1) \Gamma(x) (1-n)_{n-1} (n+1)_i}\left(\frac{2 (n+1) (1-n)_{n-1} \left(\frac{i-n+3}{2} \right)_n-(n+2) (1-n)_n \left(\frac{i-n+3}{2}\right)_{n-1}}{(n+1) (1-n)_{n-1} \left(\frac{i-n+3}{2} \right)_n}\right)^{n-n_0}\notag\\
&=\frac{(x-n_0) \Gamma (x+n)}{\Gamma (x)}=(x)_{n+1}+(n-n_0)(x)_n.
\end{align*}
A second approach to find the other independent solution follows using the next lemma.   
\begin{lemma} (\cite{kk2011}, Lemma 2, p. 3221)
Let $f$ and $g$ be two linearly independent solutions of equation
\begin{equation}\label{sec5eq7}
\Delta^n w(x)+ a_{n-1}(x)\Delta^{n-1}w(x)+\cdots +a_1(x)\Delta w(x)+a_0(x)w(x)=0\
\end{equation}
Set $u=\Delta (f/g)$. Then $w=u(x)$ satisfies
\begin{equation}\label{sec5eq8}
\Delta^{n-1} w(x)+ b_{n-1}(x)\Delta^{n-2}w(x)+\cdots +b_1(x)\Delta w(x)+b_0(x)w(x)=0
\end{equation}
where
\begin{equation}\label{sec5eq9}
b_j(x)=\sum_{k=j+1}^n {k\choose j+1}a_k(x)\frac{\Delta^{k-j-1}g(x+j+1)}{g(x+n)},\qquad j=0,1,2,\cdots,n-2.
\end{equation}
Here  we have, by convention, $a_n(x)=1$.
\end{lemma}

For $n=2$, the difference equation \eqref{sec5eq7} reads
\begin{equation}\label{sec5eq10}
\Delta^2 w(x)+a_1(x) \Delta w(x)+a_0(x) w(x)=0
\end{equation}
with $f(x)$ and $g(x)$ be two linearly independent solutions.  Then $w=u(x)=\Delta(f/g)$ satisfies the first-order difference equation
\begin{equation}\label{sec5eq11}
\Delta w(x)+b_0(x) w(x)=0
\end{equation}
 where
\begin{equation}\label{sec5eq12}
b_0(x)=a_1(x)\frac{g(x+1)}{g(x+2)}+2\frac{\Delta g(x+1)}{g(x+2)}=2+\left(a_1(x)-2\right)\frac{g(x+1)}{g(x+2)},
\end{equation}
for, by convention, $a_2(x)=1$. The solution of the first order difference equation \eqref{sec5eq11} is given by
\begin{equation}\label{sec5eq13}
w(x)=C_1\prod_{j=n_0}^{x-1}(1-b_0(j))=C_1\prod_{j=n_0}^{x-1}\left(-1-\left(a_1(j)-2\right)\frac{g(j+1)}{g(j+2)}\right)
\end{equation}
and the second independent solution $f(x)$ follows by solving the first-order inhomogeneous difference equation
\begin{equation}\label{sec5eq14}
f(x+1)-\frac{g(x+1)}{g(x)}f(x)= C_1g(x+1) \prod_{j=n_0}^{x-1}\left(\left(2- a_1(j)\right)\frac{g(j+1)}{g(j+2)}-1\right)
\end{equation} 
The solution of the equation is easily found to be 
\begin{equation}\label{sec5eq15}
f(x)=C_2\left(\prod_{i=n_0}^{x-1}\frac{g(i+1)}{g(i)}\right)+C_1\sum_{i=n_0}^{x-1}\left[\left(\prod_{\ell=i+1}^{x-1}\frac{g(\ell+1)}{g(\ell)}\right) g(i+1) \prod_{j=n_0}^{i-1}\left(\left(2- a_1(j)\right)\frac{g(j+1)}{g(j+2)}-1\right)\right] 
\end{equation}
which resemble the general solution as given by \eqref{sec3eq16}. Thus,  $a_1(x)=-2 n/(1+x)$ and $g(x)=(x)_n$, it follow that
\begin{equation}\label{sec5eq16}
f(x)=C_2\left(\prod_{i=n_0}^{x-1}\frac{i+n}{i}\right)+C_1\sum_{i=n_0}^{x-1}\left[\left(\prod_{\ell=i+1}^{x-1}\frac{\ell+n}{\ell}\right) (i+1)_n \prod_{j=n_0}^{i-1}\left(\left(2+\frac{2n}{1+j}\right)\frac{j+1}{j+1+n}-1\right)\right] 
\end{equation}
Straightforward computation shows that
\begin{equation}\label{sec5eq17}
f(x)=C(x)_n+B(x)_{n+1},
\end{equation}
where $C$ and $B$ are  unit periodic functions 
as expected and easily confirmed by direct substitution. 
\subsection{Difference equation for dual polynomials}
Let $\{Q_n(x)\}$ be a sequence of discrete orthogonal polynomials  and let
 \bea \label{sec5eq18}
 \sum_{j=0}^\infty  Q_m(x_j) Q_n(x_j) w_j = \delta_{m,n}/u_n.
 \eea
 Thus the rows of the matrix whose $(i,j)$ element is $\{Q_i(x_j) 
 \sqrt{u_i w_j}\}, i,j =0, 1, \cdots$ are orthonormal vectors. The associativity of matrix multiplication then implies that this matrix is an orthogonal matrix. This forces the columns to be orthonormal vectors, that is 
 \bea\label{sec5eq19}
 \sum_{n=0}^\infty  Q_n(x_i) Q_n(x_j) u_n = \delta_{i,j}/w_j.
 \eea
 A birth and death process with birth rates $\{\b(n)\}$ and death rates 
 $\{d(n)\}$ generates a sequence of orthogonal polynomials $\{Q_n(x)\}$.  
 The initial values are $Q_0(x)=1, Q_1(x) = (b(0)+d(0) -x)/b(0)$ and the  recurrence relation 
 \bea\label{sec5eq20}
 -xQ_n(x) = b(n)Q_{n+1}(x) + d(n) Q_{n-1}(x) - [b(n)+d(n)]Q_n(x), n >0.
 \eea
 If $\{Q_n(x)\}$ is orthogonal with respect to a discrete measure then the 
 dual polynomials $\{Q_n(x_j):j=0, 1, \cdots\}$, where now the variable 
 is $n$ and the degree is $j$ is called the polynomial dual to $\{Q_n(x)\}$. 
 There are many instances of this in the Askey scheme \cite{Koe:Swa}. The bispectral problem of Duistermaat and Gr\"unbaum \cite{DG1986} is also related to this phenomenon.
 In such cases the dual polynomials will satisfy the difference equation 
 \bea\label{sec5eq21}
 \xi y(x) = b(x)y(x+1) + d(x) y(x-1) - [b(x)+d(x)]y(x).
 \eea
 In other words 
 \bea\label{sec5eq22}
 b(x+1) \Delta^2y(x) + [b(x+1)-d(x+1) -\xi] \Delta y(x) - \xi y(x)=0. 
 \eea
The case of birth and death process polynomials when  $b(x)$ and $d(x)$ are polynomials of degree at most $2$ and 
$b(x)-d(x)$ is of degree at most $1$ was studied in \cite{Ism:Let:Val}, where their orthogonality measure was also constructed.  Their dual polynomials will then satisfy the hypergeometric difference equation  
\begin{eqnarray}\label{sec5eq18}
({a_2x^2+a_1x+a_0})\Delta^2y(x)+({b_1 x+b_0}) \Delta y(x)-ky(x)=0,
\end{eqnarray}
This equation may also be considered as a difference analogue of the hypergeometric difference equation. 
\begin{eqnarray}\label{sec5eq19}
\lambda_0(x)=- \frac{b_1 x+b_0}{a_2x^2+a_1x+a_0} ,\qquad s_0(x)=\frac{k}{a_2x^2+a_1x+a_0},
\end{eqnarray}
it follow, by the recursive evaluation of the DAIM sequence, that 
the termination condition  $\delta_n(x)=\lambda_n(x)s_{n-1}(x)-\lambda_{n-1}(x)s_n(x),n=1,2,\dots$ yields
\begin{align*}
\delta_1(x)&=\dfrac{b_1-k}{a_0+(1+x) (a_1+a_2(1+x))}\,\delta_0(x)\\
\delta_2(x)&=\dfrac{2a_2+2b_1-k}{a_0+(2+x)(a_1+a_2(2+x))}\,\delta_1(x)\\
\delta_3(x)&=\dfrac{6 a_2+3 b_1-k}{a_0+(3+x)(a_1+a_2(3+x))}\,\delta_2(x)\\
\delta_4(x)&=\dfrac{12 a_2+4 b_1-k}{a_0+(4+x)(a_1+a_2(4+x))}\,\delta_3(x).
\end{align*} 
For arbitrary $n$, it is not difficult to show that, for $n=1,2,\dots$,
\begin{eqnarray}\label{sec5eq20}
\delta_n(x)&=\dfrac{n(n-1) a_2+n b_1-k}{a_0+(n+x)(a_1+a_2(n+x))}\,\delta_{n-1}(x)=\dfrac{\prod_{j=0}^n j(j-1) a_2+j b_1-k}{\prod_{j=0}^n a_0+(j+x)(a_1+a_2(j+x))}\,\delta_0(x)\end{eqnarray} 
where $\delta_0(x)=s_0(x)$ given $\lambda_{-1}(x)=-1,s_{-1}(x)=0.$ Clearly, for $\delta_{n-1}(x)\neq 0$, $\delta_n(x)=0$ only if 
\begin{eqnarray}\label{sec5eq21}
k=n(n-1)\,a_2+n\,b_1,\qquad n=1,2,\dots.
\end{eqnarray} 
In this case, the polynomial solutions of the difference equation
\begin{eqnarray}\label{sec5eq22}
\Delta^2y(x)=-\dfrac{b_1 x+b_0}{a_2x^2+a_1x+a_0}\, \Delta y(x)+\dfrac{n(n-1)\,a_2+n\,b_1}{a_2x^2+a_1x+a_0}y(x),
\end{eqnarray}
are given as
\begin{itemize}
\item For $n=0$, $y_0(x)=1.$
\item For $n=1$, 
\begin{align*}
y_1(x)=\prod_{i=x_0}^{x-1}\left[1-\dfrac{s_0(x)}{\lambda_0(x)}\right]=x + \dfrac{b_0}{b_1}.
\end{align*}
\item For $n=2$, 
\begin{align*}
y_2(x)&=\prod_{i=x_0}^{x-1}\left[1-\frac{s_1(x)}{\lambda_1(x)}\right]=x^2 + \dfrac{(2 a_1 + 2 b_0 + b_1)}{(2 a_2 + b_1)}\,x+\dfrac{ (
 a_0 (2 a_2 + b_1) + b_0 (a_1 + a_2 + b_0 + b_1))}{((a_2 + b_1) (2 a_2 + b_1))}.
\end{align*}
\item For $n=3$, 
\begin{align*}
y_3(x)&=x^3 +\dfrac{ 3 (2 a_1 + 2 a_2 + b_0 + b_1)}{(4 a_2 + b_1)}\,
  x^2 \nonumber\\
  &+ \frac{(6 a_1^2 + 12 a_2 b_0 + 3 b_0^2 + 5 a_2 b_1 + 6 b_0 b_1 + 2 b_1^2 + 
     3 a_0 (4 a_2 + b_1)+ 
     9 a_1 (2 a_2 + b_0 + b_1))}{(3 a_2 + b_1) (4 a_2 + b_1)}\,x
  \nonumber\\
  & +(a_0 (36 a_2^2 + 10 a_2 b_0 + 24 a_2 b_1 + 3 b_0 b_1 + 4 b_1^2 + 
       4 a_1 (3 a_2 + b_1))\\
       & + 
    b_0 (2 a_1^2 + 10 a_2^2 + 7 a_2 b_0 + b_0^2 + 9 a_2 b_1 + 3 b_0 b_1 + 
       2 b_1^2\\
       & + a_1 (12 a_2 + 3 b_0 + 5 b_1))/({(2 a_2 + b_1) (3 a_2 + 
      b_1) (4 a_2 + b_1)}).
\end{align*}
and so on for higher order.
\end{itemize}
As special cases of the hypergeometric difference equation (\ref{sec5eq18}) are the Meixner difference equation
\begin{eqnarray}\label{sec5eq23}
\Delta^2 y(x)=-\dfrac{(\mu-1)(x-n+1)+\mu\,\delta}{\mu\,(x+\delta+1)}\,\Delta y(x) - \dfrac{k}{\mu\,(x+\delta+1)} y(x),
\end{eqnarray}
and the Hermite difference equation
\begin{eqnarray}\label{sec5eq24}
\Delta^2 y(x)=(a\,x+b)\,\Delta y(x) + \gamma\, y(x).
\end{eqnarray}

\section{$q$-Asymptotic Iteration Method ($q$AIM)}
We consider the linear second-order $q$-difference equation
\begin{eqnarray}\label{qdeq1}
D_q^2 y(x)&=\lambda_0(x)D_q y(x)+s_0(x) y(x).
\end{eqnarray} 
In general, we have
\begin{eqnarray}\label{qdeq2}
D_q^{n+2} y(x)&=\lambda_{n}(x)D_q y(x)+s_n(x)y(x), 
\end{eqnarray} 
where the functions $\l_n(x)$ and $s_n(x)$ are generated by 
\bea\label{qdeq3}
\lambda_n(x)=D_q\lambda_{n-1}(x)+\lambda_{n-1}(qx)\lambda_0(x)+s_{n-1}(qx), \qquad s_n(x)=D_qs_{n-1}(x)+\lambda_{n-1}(qx)s_0(x).
\eea
If the termination condition 
\begin{eqnarray}\label{qdeq4}
\frac{s_n(x)}{\lambda_n(x)}=\frac{s_{n-1}(x)}{\lambda_{n-1}(x)}.
\end{eqnarray}
holds for some $n$ then 
\begin{eqnarray}\label{qdeq5}
\dfrac{D_q^{n+2} y(x)}{D_q^{n+1} y(x)}&=\dfrac{\lambda_{n}(x)D_q y(x)+s_n(x)y(x)}
{\lambda_{n-1}(x)D_q y(x)+s_{n-1}(x)y(x)}=\dfrac{\lambda_{n}(x)}{\lambda_{n-1}(x)}.
\end{eqnarray} 
Equation (\ref{qdeq5}) can be written as
\begin{eqnarray}\label{qdeq6}
D_q\left(D_q^{n+1} y(x)\right)&
=\dfrac{\lambda_{n}(x)}{\lambda_{n-1}(x)}{D_q^{n+1} y(x)}
\end{eqnarray} 
 This is a first-order $q$-difference equation in $D_q^{n+1} y(x)$ and according to (\ref{sec2eq16})-(\ref{sec2eq17}) its solution is 
\begin{eqnarray}\label{qdeq7}
D_q^{n+1} y(x)
&= D_q^{n+1}y(0)\prod_{k=0}^{\infty}\left[1-(1-q)q^k x
\frac{\lambda_n(q^kx)}{\lambda_{n-1}(q^kx)}\right]^{-1}.
\end{eqnarray}
The infinite product will converge if the ratio $\l_n(x)/\l_{n-1}(x)$ is bounded in a neighbourhood of $x=0$ in the complex plane. On the other hand (\ref{qdeq2}) implies 
\begin{eqnarray}\label{qdeq8}
\lambda_{n-1}(x)D_q y(x)+s_{n-1}(x)y(x)
&=\dfrac{D_q^{n+1}y(0)}{\prod_{k=0}^{\infty}\left[1-(1-q)q^k x\dfrac{\lambda_n(q^kx)}{\lambda_{n-1}(q^kx)}\right]}
\end{eqnarray}
or equivalently 
\begin{eqnarray}\label{qdeq9}
D_q y(x)=-\dfrac{s_{n-1}(x)}{\lambda_{n-1}(x)}y(x)+\dfrac{D_q^{n+1}y(0)}{\lambda_{n-1}(x)}{\prod_{k=0}^{\infty}\left[1-(1-q)q^k x\dfrac{\lambda_n(q^kx)}{\lambda_{n-1}(q^kx)}\right]^{-1}}. 
\end{eqnarray}
In view of (\ref{sec2eq18})--(\ref{sec2eq19}) the solution of the original second-order $q$-difference equation (\ref{qdeq1})  is given by 
\bea
\label{qdeq10}
\begin{gathered}
y(x) =\dfrac{y(0)}{\prod_{k=0}^{\infty}\left[1+(1-q)q^k x\dfrac{s_{n-1}(q^kx)}{\lambda_{n-1}(q^kx)}\right]}\\
+D_q^{n+1}y(0)
\sum_{k=0}^{\infty}\frac{\dfrac{(1-q)q^{k}x}
{\lambda_{n-1}(q^kx)}}{\prod_{i=0}^{\infty}
\left[1-(1-q)q^{i+k} x\dfrac{\lambda_n(q^{i+k}x)}
{\lambda_{n-1}(q^{i+k}x)}\right]\prod_{j=0}^{k}
\left[1-(1-q)q^{j}x\dfrac{s_{n-1}(q^{j}x)}{\lambda_{n-1}(q^{j}x)}\right]}. 
\end{gathered}
\eea
It is known that $y_1$ and $y_2$ are linearly independent 
if and only if the determinant 
\begin{equation} \label{qdeq11}
\left| \begin{array}{cc}
y_1(x) & D_q y_1(x) \\
y_2(x) & D_qy_2(x)
\end{array}\right| \ne 0, 
\end{equation}
for all $x$ in the domain of definition.  It is easy to see that this is case 
with the two solutions given above.

\section{Implementation and Examples}
Our first example is the $q$-Laguerre polynomials,
 \cite[p. 109]{Koe:Swa}. 
They satisfy the $q$-Difference equation:
\begin{eqnarray} \label{qdeq12}
\left(1+q^{\eta}+q^{\eta+n} x\right)y(x) 
= q^\eta (1 + x)\,y(q\,x)  + y(q^{-1}x),
\end{eqnarray}
It is easy to write this equation in the form 
\begin{eqnarray} \label{qdeq13}
\begin{gathered}
D_q^2y(x)=\left(\frac{q^{-1 - \eta} -1 -(1+q-q^{n}) x}{(q-1)\,x\, (1+q\, x)}\right)D_qy(x)+\left(\dfrac{q^n-1}{(q-1)^2\, x\, (1+q\, x)}\right)y(x), 
\end{gathered}
\end{eqnarray}
with 
\bea
 \label{qdeq14}
\lambda_0(x)= \frac{q^{-1 - \eta} -1 -(1+q-q^{n}) x}{(q-1)\,x\, (1+q\, x)}, 
\quad s_0(x) =\dfrac{q^n-1}{(q-1)^2\, x\, (1+q\, x)}
\eea
Using (\ref{qdeq6}) and the definition 
\begin{eqnarray}  \label{qdeq15}
\delta_m(x)=\lambda_m(x)s_{m-1}(x)-\lambda_{m-1}(x)s_{m}(x), 
\qquad m=1,2,\dots
\end{eqnarray}
it follow that
\begin{eqnarray*}\begin{aligned}
\delta_1&=\dfrac{(q - q^n) (q^n-1)}{x^2(qx+1)(q^2x+1)(q-1)^4},\\
\delta_2&=\dfrac{(q - q^n) (q^2 - q^n) (q^n-1)}{x^3(qx+1)(q^2x+1)(q^3x+1)(q-1)^6},\\
\delta_3&=\dfrac{(q - q^n) (q^2 - q^n) (q^3 - q^n) (q^n-1)}{x^4(qx+1)(q^2x+1)(q^3x+1)(q^4x+1)(q-1)^8},\\
\delta_4&=\dfrac{(q - q^n) (q^2 - q^n) (q^3 - q^n) (q^4 - q^n) ( q^n-1)}{x^5(qx+1)(q^2x+1)(q^3x+1)(q^4x+1)(q^5x+1)(q-1)^{10}},\\
\delta_5&=\dfrac{(q - q^n) (q^2 - q^n) (q^3 - q^n) (q^4 - q^n) (q^5 - q^n) (q^n-1)}{x^6(qx+1)(q^2x+1)(q^3x+1)(q^4x+1)(q^5x+1)(q^6x+1)(q-1)^{12}}. 
\end{aligned}\end{eqnarray*}
In general we observe the pattern
\begin{eqnarray}  \label{qdeq16}
\delta_{m+1}=\dfrac{q^{m+1}-q^n}{(q-1)^2x(1+q^{m+2})}\delta_{m},\qquad m=0,1,2,\dots 
\end{eqnarray}
which has been tested up to $m = 15$. 
Based on this we conclude that  $\delta_m=0$ if and only if $m=n$.  For an exact solution, we use the following expression:
\begin{eqnarray}
\label{qdeq17}
y_n(x)
&=\dfrac{y_n(0)}{\prod\limits_{k=0}^{\infty}\left[1+(1-q)q^k x\dfrac{s_{n-1}(q^kx)}{\lambda_{n-1}(q^kx)}\right]}.
\end{eqnarray}
For example, the polynomial solution of degree $5$ is 

\begin{eqnarray}\label{qdeq18}
\begin{gathered}
y_5(x)
= y_5(0) \prod\limits_{k=0}^{\infty}\left[1+(1-q)\,q^k x\,\frac{s_{4}(q^kx)}
{\lambda_{4}(q^kx)}\right]^{-1}\\
=y_5(0)\left(1+\frac{q^{1+\eta} \left(1-q^5\right)}
{(1-q)(q^{1+\eta}-1)}\right.\,x
\left.+\frac{q^{4 + 2 \eta} (1 + q^2) (1 - q^5))}
{(1-q)(q^{ 1 + \eta}-1) (q^{2 + \eta}-1)}\,x^2\right.\\
+\frac{q^{9+3 \eta} \left(1+q^2\right) \left(1-q^5\right)}
{(1-q)\left(q^{1+\eta}-1\right) \left(q^{2+\eta}-1\right) \left(q^{3+\eta}-1
\right)}\,x^3\\
\left.+\frac{q^{4 (4+ \eta)} \left(1-q^5\right)}
{(1-q)\left(q^{1+ \eta}-1\right) \left(q^{2+\eta}-1\right) 
\left(q^{3+\eta}-1\right) \left(q^{4+\eta}-1\right)}\, x^4\right.\\
+\left.\frac{q^{5 (5+\eta)}}{\left(q^{1+\eta}-1\right)
 \left(q^{2+\eta}-1\right) \left(q^{3+\eta}-1\right)
  \left(q^{4+\eta}-1\right) \left(q^{5+\eta}-1\right)}\,x^5\right). 
  \end{gathered}
  \end{eqnarray}

\noindent More importantly we can also write down a second solution to the $q$-difference equation. It is know that the second solution is related to the function of the second kind, see \cite{Ismbook}, \cite{Ism:Man}. 
\vskip0.1true in
\noindent Our second example is the Al-Salam-Carlitz polynomials $\{U_n(x)\}$, 
\cite{Ismbook}, \cite{Koe:Swa}. Their $q$-Difference equation is 
\begin{eqnarray}
\label{qdeq19}\begin{gathered}
aq^{n-1}y(q^2\,x)=\left(a q^{-1+n}+a q^n-(1+a) q^{1+n} x+q^2 x^2\right) y(qx)-q^n(1-qx)(a-qx)y(x).\end{gathered}
\end{eqnarray}
Thus
\begin{eqnarray}
\label{qdeq20}
D_q^2 y(x)=\left(\frac{q+a q-q^{2-n} x}{a-a q}\right)D_qy(x)-\frac{q^{2-n} \left(-1+q^n\right)}{a (-1+q)^2}y(x)
\end{eqnarray}
The termination condition $\delta_n(x)=\lambda_n(x)s_{n-1}(x)-s_n(x)\lambda_{n-1}(x)\equiv 0$, $n=1,2\dots$
where $\{\lambda_n(x)\}$ and $\{s_n\}$ satisfy, see (\ref{qdeq3}),
\begin{align*}
 \delta_1(x) &= \dfrac{q^{2(2-n)}(q^n-1) (q - q^n)}{a^2(q-1)^4}=\dfrac{q^{2-n}(q-q^n)}{a(q-1)^2}\delta_0(x), \\
\delta_2(x)&=\dfrac{q^{3(2-n)}(q^n-1)(q - q^n) (q^2 - q^n)}{a^3(q-1)^6}=\dfrac{q^{2-n}(q^2-q^n)}{a(q-1)^2}\delta_1(x),\\
 \delta_3(x)&=\dfrac{q^{4(2-n)}(q^n-1)(q - q^n) (q^2 - q^n)(q^3-q^n)}{a^4(q-1)^8}=\dfrac{q^{2-n}(q^3-q^n)}{a(q-1)^2}\delta_2(x),\\
\delta_4(x)&=\dfrac{q^{5(2-n)}(q^n-1)(q - q^n) (q^2 - q^n)(q^3-q^n)(q^4-q^n)}{a^5(q-1)^{10}}=\dfrac{q^{2-n}(q^4-q^n)}{a(q-1)^2}\delta_3(x),\\
\delta_5(x)&=\dfrac{q^{6(2-n)}(q^n-1)(q - q^n) (q^2 - q^n)(q^3-q^n)(q^4-q^n)(q^5-q^n)}{a^6(q-1)^{12}}=\dfrac{q^{2-n}(q^5-q^n)}{a(q-1)^2}\delta_4(x).
\end{align*}
We may then observe the pattern  
\begin{eqnarray}
\label{qdeq21}
\delta_{m+1}(x)=\dfrac{q^{2-n}(q^{m+1}-q^n)}{a(q-1)^2}\delta_m,
\quad m=1,2,\dots.
\end{eqnarray}
We verified this pattern up to $m = 15$. Thus the smallest $m$ which makes 
$\delta_m(x)=0$ is $m=n$. 
The polynomials solution is then given by (\ref{qdeq17}). 
 For example the polynomial of order five is given by
{\small
\begin{eqnarray*}
\begin{gathered}
y_5(x)/y_5(0) 
= 1  
-\frac{(1+q+q^2+q^3+q^4) ((1+a^4) q^4+a q (1+q) 
(1+q^2)(1+a^2)+a^2 (1+q^2) (1+q+q^2))}
{(1+a) q^2 (q^6+a^4 q^6+a q^2 (1+q) (1+q^2)+a^3 q^2 (1+q) 
(1+q^2)+a^2 (1+q^2) (1+q+q^4))}\,x \qquad  \qquad \\
+\frac{(1+q^2) (a+a q+(1+a^2) q^2) (1+q
+q^2+q^3+q^4)}{q^3 (q^6+a^4 q^6+a q^2 (1+q) (1+q^2)
+a^3 q^2 (1+q) (1+q^2)+a^2 (1+q^2) 
(1+q+q^4))}\,x^2 \qquad\\
 -\frac{(a+(1+a+a^2) q) (1+q^2)
 (1+q+q^2+q^3+q^4)}{(1+a) q^4 (q^6+a^4 q^6+a q^2 (1+q) 
 (1+q^2)+a^3 q^2 (1+q) (1+q^2)+a^2 (1+q^2) 
 (1+q+q^4))}\,x^3 \qquad \qquad \\
+\frac{1+q+q^2+q^3+q^4}{q^4 (q^6+a^4 q^6+a q^2 (1+q) 
(1+q^2)+a^3 q^2 (1+q) (1+q^2)+a^2 (1+q^2) 
(1+q+q^4))}\,x^4 \qquad\\
.-\frac{x^5}{(1+a) q^4 (q^6+a^4 q^6+a q^2 (1+q) (1+q^2)
+a^3 q^2 (1+q) (1+q^2)+a^2 (1+q^2)
 (1+q+q^4)}.\qquad 
 \end{gathered}
 \end{eqnarray*}
}

 Our third example is the Stieltjes-Wigert $q$-difference equation, 
\cite[page 116]{Koe:Swa2}. 
The $q$-Difference equation satisfied by the Stieltjes--Wigert 
polynomials  is
\begin{eqnarray}\label{qdeq22}
-x(1-q^n)y(x)&=xy(q\,x)-(1+x)y(x)+y(q^{-1}x),
\end{eqnarray}
which has the equivalent form 
\begin{eqnarray}\label{qdeq23}
D_q^2 y(x)&=\left(\dfrac{1-q \left(1+q-q^n\right) x}{(q-1) q^2 x^2}\right)D_q y(x)+\dfrac{q^n-1}{(q-1)^2 q x^2}y(x).
\end{eqnarray} 
Using the recursion \ref{qdeq3} with 
\begin{eqnarray}\label{qdeq24}
\lambda_0(x) = \left(\frac{1-q \left(1+q-q^n\right) x}{(q-1) q^2 x^2}\right), 
\quad s_0(x)= \frac{q^n-1}{(q-1)^2 q x^2},
\end{eqnarray}
we find that 
\begin{align*}
  \delta_1(x)&= \dfrac{(q^n-1) (q^n-q)}{q^3x^4(q-1)^4}=\dfrac{q^n-q}{q^2x^2(q-1)^2}\delta_0(x),\quad \\
  \delta_2(x)&= \dfrac{(q^n-1) (q^n-q)(q^n-q^2)}{q^6x^6(q-1)^6}=\dfrac{q^n-q^2}{q^3x^2(q-1)^2}\delta_1(x),\\
  \delta_3(x)& =\dfrac{(q^n-1) (q^n-q)(q^n-q^2)(q^n-q^3)}{q^{10}x^8(q-1)^8}=\dfrac{q^n-q^3}{q^4x^2(q-1)^2}\delta_2(x),\\
  \delta_4(x)&=\dfrac{(q^n-1) (q^n-q)(q^n-q^2)(q^n-q^3)(q^n-q^4)}{q^{15}x^{10}(q-1)^{10}}=\dfrac{q^n-q^4}{q^5x^2(q-1)^2}\delta_3(x),\\
 \delta_5(x)& =\dfrac{(q^n-1) (q^n-q)(q^n-q^2)(q^n-q^3)(q^n-q^4)(q^n-q^5)}{q^{21}x^{12}(q-1)^{12}}=\dfrac{q^n-q^5}{q^6x^2(q-1)^2}\delta_4(x).
\end{align*}
This suggests the pattern 
\begin{eqnarray}\label{qdeq25}
\delta_{m+1} =\dfrac{q^n-q^{m+1}}{q^{m+2}x^2(q-1)^2}\delta_m(x),\qquad m=0,1,2,\dots.
\end{eqnarray}
Again, we verified this up to $m = 15$. Thus the smallest $m$ for which 
$\delta_m(x)=0$ is $m=n$. 
The polynomial  solutions are then given by (\ref{qdeq17}). 
The fifth-order polynomial solution is given by
\begin{eqnarray} \nonumber 
\begin{gathered}
y_5(x)=y_5(0)(1- \left(1+q+q^2+q^3+q^4\right)\,(qx) + 
 \left(1+q^2\right) \left(1+q+q^2+q^3+q^4\right)\,(q^2x)^2 
 \\
-\left(1+q^2\right) \left(1+q+q^2+q^3+q^4\right)\,(q^3x)^3+\left(1+q
+q^2+q^3+q^4\right)\,(q^4x)^4-(q^{5}\,x)^5.
\end{gathered}
\end{eqnarray}
This can written in the form 
\begin{eqnarray}  
\label{qdeq26}
\begin{gathered}
\frac{y_5(x)}{y_5(0)}
= 1+\frac{(1-q^5)}{1-q)}q(-x) + \frac{(1-q^5)(1-q^4)}{(1-q)(1-q^2)}
\; q^4 (-x)^2  
 \\
+ \frac{(1-q^5)(1-q^4)}{(1-q)(1-q^2)}\; q^9 (-x)^3   + \frac{(1-q^5)}{(1-q)}\; q^{16} (-x)^4  
+ q^{25}(-x)^5.
\end{gathered}
\end{eqnarray}
From this pattern the following pattern is clear
\begin{eqnarray} 
\frac{y_n(x)}{y_n(0)} = \sum_{k=0}^n \frac{(q;q)_n}{(q;q)_k(q;q)_{n-k}}\; (-1)^k q^{k^2}x^k,
\end{eqnarray}
which can then be proved rigorously. 

\begin{rem}
It is important to note that it is not surprising that $\delta_m(x)$ for the 
$q$-Laguerre $\{L_n^{(\eta)}(x;q)\}$ and the Stieltjes--Wigert 
polynomials $\{S_n(x;q)\}$ are almost identical. The reason is that 
the part of 
$\delta_n(x)$ for the $q$-Laguerre polynomials which vanishes does not depend on $\eta$, 
and $S_n(x;q)$ and $L_n^{(\eta)}(x;q)$ 
\begin{eqnarray}\label{qdeq27}
S_n(x;q)  = \lim_{\eta \to \infty} L_n^{(\eta)}(xq^{-\eta};q)
\end{eqnarray}
\end{rem}

\setcounter{equation}{0}
\section{Limitations of DAIM and $q$-AIM}
In this section we show the limitations of the both DAIM and $q$-AIM by 
applying it to the case of linear second-order difference equation with constant coefficients. The case of linear second-order differential  equation with constant coefficients is similar. Consider the difference equation
\begin{equation}\label{sec5eq1}
\Delta^2 y(x)=a\, \Delta y(x) + b\, y(x)
\end{equation}
where $\lambda_0(x)\equiv a$ and $s_0(x)\equiv b$ are polynomials in $a$ and $b$. Therefore, the DAIM sequences (\ref{sec3eq5}) yields
\begin{eqnarray*}
\lambda_1(x)=a^2+b, \quad 
s_1(x) =ab, \quad \lambda_2(x)=a(a^2+2b),\quad 
s_2(x)=b(a^2+b)..
\end{eqnarray*}
In the present case the recurrence relations  \eqref{sec3eq5} become 
\bea
\label{eq42}
\lambda_n=
\lambda_{n-1}\lambda_0 +s_{n-1}, \qquad 
s_n= \lambda_{n-1}s_0.
\eea
it follows that   $\lambda_n$ and  $s_n$  are polynomials in $a$ and $b$ of total degree $n+1$.  
The termination condition 
 $$\frac{s_{n}(x)}{\lambda_n(x)} =\frac{s_{n-1}(x)}{\lambda_{n-1}(x)}\quad \textup{implies} \quad 
\dfrac{ \lambda_{n-1}s_0}{\lambda_{n-1}\lambda_0+s_{n-1}} =\dfrac{s_{n-1}}{\lambda_{n-1}},
$$
which leads to the quadratic equation
$$
\left(\frac{s_{n-1}(x)}{ \lambda_{n-1}(x)}\right)^2+\lambda_0(x)\left( \frac{s_{n-1}(x)}{\lambda_{n-1}(x)}\right)-s_0(x)=0
$$
with solutions
\bea
\label{eqCC1}
 \frac{s_{n-1}(x)}{ \lambda_{n-1}(x)} =\frac{-a\pm\sqrt{a^2+4b}}{2}.
\eea
It is now  clear that there is no $n$ for which (\ref{eqCC1}) holds because its left-hand side is a rational function in $a$ and $b$ but its right-hand side is an algebraic non-rational function. What is surprising is 
that the method nevertheless gives the correct answer. Indeed  
this gives  the solutions 
\begin{eqnarray}
\begin{gathered}
y_{+}(x)=\prod_{i=n_0}^{x-1} \left[1-\frac{s_{x-1}(i)}{\lambda_{x-1}(i)\,} \right]=C_1\left[1+\dfrac{a-\sqrt{a^2+4b}}{2} \right]^x,\\
y_{-}(x)= \prod_{i=n_0}^{x-1} \left[1-\frac{s_{x-1}(i)}{\lambda_{x-1}(i)\,} \right]=C_2 \left[1+\dfrac{a+\sqrt{a^2+4b}}{2} \right]^x.
\end{gathered}
\end{eqnarray}
Surprisingly, this is the correct answer, \cite{Jor}, \cite{Mil}.
\vskip0.1true in
\noindent  One is tempted to use (\ref{eq42}) to get 
 \bea
 \nonumber
 \dfrac{s_n}{\lambda_n} = \frac{b}{a+ \dfrac{s_{n-1}}{\lambda_{n-1}}},
 \eea
 which when iterated leads to the continued fraction \cite{Jon:Thr}
\bea
 \nonumber
 \frac{s_n}{\lambda_n} = \frac{b}{a}\,  {}_+ \frac{b}{a} \, {}_+ \cdots. 
 \eea
Here again we face issues of rigor because the above continued 
fraction is a periodic continued fraction and will converge to a unique 
value involving the minimal solution, via Pincherle's theorem 
\cite{Jon:Thr}. So even formally we get only one solution. 

There is also inherent inconsistency in applying AIM, DAIM, or $q$-AIM 
to equations with constant coefficients. In all cases it has been been 
proved that the terminating condition holds if and only if the equation
in question has a polynomial solution. This automatically excludes all 
equations with constant coefficients, except trivial ones like 
$y^{\prime\prime} =0, \Delta^2\, y(x) =0, D_q^2\, y(x)=0$. This also invalidates the application of AIM to general Euler equations of the type
\bea
\nonumber
x^2\, y^{\prime\prime}(x) + a\, x\, y^{\prime}(x) + b\, y(x) =0,
\eea
What is a surprise is that this invalid applications of the AIM, DAIM, or 
$q$-AIM technique give the correct answers. 
 
\section{Acknowledgments}
\medskip
\noindent Partial financial support of this work under Grant No. GP249507 from the 
Natural Sciences and Engineering Research Council of Canada
 is gratefully acknowledged. 

\end{document}